\documentclass[11pt]{article}
\usepackage[utf8]{inputenc}
\usepackage{lmodern}
\usepackage{subfiles}
\usepackage{enumitem}
	\setenumerate{label={\normalfont(\roman*)}, itemsep=0em} 
        
\usepackage{amsfonts}
\usepackage{amsthm}
\usepackage{amsmath}
\usepackage{amssymb}
\usepackage{amscd}
\usepackage{mathrsfs}
\usepackage{mathtools}
\usepackage{bm}
\usepackage{esint}

\usepackage[margin=3cm]{geometry}
\usepackage{indentfirst}
\usepackage{graphicx}
\usepackage{graphics}
\usepackage{lscape}
\usepackage{tikz-cd}
\usepackage{color}
\usepackage{pict2e}
\usepackage{epic}
\usepackage{epstopdf}
\usepackage{titlesec, titlefoot}
	\titleformat{\section}[block]{\Large\bfseries\filcenter}{\thesection}{1em}{}
\usepackage{commath}
\usepackage{float}
\usepackage{caption}
\usepackage{etoolbox}
\usepackage[affil-it]{authblk}
\usepackage{combelow}

\usepackage[hidelinks]{hyperref}
\hypersetup{bookmarksopen=true} 
\usepackage{hypcap}

\graphicspath{{./Pictures/}}
\allowdisplaybreaks

\usepackage{listings}
\usepackage{fancybox}

\theoremstyle{plain}

\newcommand{\thistheoremnames}{}
\newtheorem*{genericthms}{\thistheoremnames}
\newenvironment{bigthm*}[1]
  {\renewcommand{\thistheoremnames}{#1}%
   \begin{genericthms}}
  {\end{genericthms}}

\renewcommand*\thesection{\arabic{section}}
\numberwithin{equation}{section}

\theoremstyle{plain}
\newtheorem{thm}{Theorem}
\newtheorem{lemma}[thm]{Lemma}
\newtheorem{prop}[thm]{Proposition}
\newtheorem{cor}[thm]{Corollary}
\numberwithin{thm}{section} 

\theoremstyle{definition}
\newtheorem{ndef}[thm]{Definition}

\newtheorem{question}[thm]{Question}

\newtheorem{remark}[thm]{Remark}

\newcommand{\thistheoremname}{}
\newtheorem{genericthm}[equation]{\thistheoremname}

\expandafter\let\expandafter\oldproof\csname\string\proof\endcsname
\let\oldendproof\endproof
\renewenvironment{proof}[1][\proofname]{%
  \oldproof[\upshape \bfseries #1:]%
}{\oldendproof}

\makeatletter
\def\@makechapterhead#1{%
  \vspace*{50\p@}%
  {\parindent \z@ \raggedright \normalfont
    \interlinepenalty\@M
    \Huge\bfseries  \thechapter.\quad #1\par\nobreak
    \vskip 40\p@
  }}
\makeatother

\def \a{\alpha}
\def \R {\mathbb{R}}
\def \Z {\mathbb{Z}}

\def \N{\mathbb{N}}
\def \D{\textup{D}}
\def \J{\textup{J}}

\def \e{\varepsilon}
\def \d{\,\textup{d}}

\def \exc{\backslash}
\def \p{\partial}

\def \mc{\mathcal}
\def \mb{\mathbb}

\def \w{\rightharpoonup}

\def \tp{\textup}

\makeatletter
\DeclareFontFamily{OMX}{MnSymbolE}{}
\DeclareSymbolFont{MnLargeSymbols}{OMX}{MnSymbolE}{m}{n}
\SetSymbolFont{MnLargeSymbols}{bold}{OMX}{MnSymbolE}{b}{n}
\DeclareFontShape{OMX}{MnSymbolE}{m}{n}{
    <-6>  MnSymbolE5
   <6-7>  MnSymbolE6
   <7-8>  MnSymbolE7
   <8-9>  MnSymbolE8
   <9-10> MnSymbolE9
  <10-12> MnSymbolE10
  <12->   MnSymbolE12
}{}
\DeclareFontShape{OMX}{MnSymbolE}{b}{n}{
    <-6>  MnSymbolE-Bold5
   <6-7>  MnSymbolE-Bold6
   <7-8>  MnSymbolE-Bold7
   <8-9>  MnSymbolE-Bold8
   <9-10> MnSymbolE-Bold9
  <10-12> MnSymbolE-Bold10
  <12->   MnSymbolE-Bold12
}{}

\let\llangle\@undefined
\let\rrangle\@undefined
\DeclareMathDelimiter{\llangle}{\mathopen}%
                     {MnLargeSymbols}{'164}{MnLargeSymbols}{'164}
\DeclareMathDelimiter{\rrangle}{\mathclose}%
                     {MnLargeSymbols}{'171}{MnLargeSymbols}{'171}
\makeatother

\begin{document}

\title{\textbf{The Dirichlet problem for the Jacobian equation\\
 in critical and supercritical Sobolev spaces}}

\author[1]{{\Large Andr\'e Guerra}}
\author[1]{{\Large Lukas Koch}}
\author[2]{{\Large Sauli Lindberg}}

\affil[1]{\small University of Oxford, Andrew Wiles Building Woodstock Rd, Oxford OX2 6GG, United Kingdom 
\protect \\
  {\tt{\{andre.guerra, lukas.koch\}@maths.ox.ac.uk}}
  \vspace{1em} \ }

\affil[2]{\small 
Aalto University, Department of Mathematics and Systems Analysis, P.O. Box 11100, FI-00076 Aalto, Finland 
\protect\\
  {\tt{sauli.lindberg@aalto.fi}} \ }

\date{}

\maketitle

\begin{abstract}
We study existence and regularity of solutions to the Dirichlet problem for the prescribed Jacobian equation, $\det \D u =f$, where $f$ is integrable and bounded away from zero. In particular, we take $f\in L^p$, where $p>1$, or in $L\log L$. We prove that for a Baire-generic $f$ in either space there are no solutions  with the expected regularity.
 \end{abstract}

\unmarkedfntext{
\hspace{-0.8cm}
\emph{Acknowledgements.} A.G. and L.K. were supported by the Engineering and Physical Sciences Research Council [EP/L015811/1]. S.L. was supported by the AtMath Collaboration at the University of Helsinki and the ERC grant 834728-QUAMAP.
}

\section{Introduction}

For $n>1$ let us take a bounded domain $\Omega\subset \R^n$ which is both smooth and uniformly convex.
We consider the prescribed Jacobian equation, that is,
\begin{equation}
\det \D u = f \quad \tp{ in } \Omega,
\label{eq:jacPDE}
\end{equation}
which is a first-order fully non-linear underdetermined PDE. The equation \eqref{eq:jacPDE} has a rich geometric flavour: formally, it can be written as
$u^* (f \d y)= \d x$,
where $u^*$ denotes the pull-back of the volume form $f(y) \d y$ under $u$. Hence \eqref{eq:jacPDE} can be read as asserting equivalence of volume forms over $\Omega$ with the same mass.

To be precise, we will consider a Dirichlet problem problem for \eqref{eq:jacPDE}:
\begin{equation}
\label{eq:jac}
\begin{cases}
\det{\D u}=f & \tp{in } \Omega,\\
u=\tp{id} & \tp{on } \p \Omega.
\end{cases}
\end{equation}
Here $f\colon \Omega\to \R$ is integrable and  satisfies the compatibility and non-degeneracy conditions
\begin{equation}
\label{eq:hypdata}
\frac{1}{|\Omega|}\int_\Omega f(x)\d x = 1 \qquad \tp{ and } \qquad \inf_\Omega f\geq c> 0.
\end{equation}

Whenever $f$ is smooth or, more generally, H\"older continuous, it is natural to look for solutions which are at least $C^1$ and thus \eqref{eq:jac} can be understood classically. Under these assumptions, there is a very complete well-posedness theory going back to the seminal works of \textsc{Moser} \cite{Moser1965} and \textsc{Dacorogna}--\textsc{Moser} \cite{Dacorogna1990}, see also the book \cite{Csato2012} and the references therein.
An alternative approach to this problem was taken in \cite{Carlier2012}, where the well-posedness theory for \eqref{eq:jac} is derived from \textsc{Caffarelli}'s regularity theory \cite{Caffarelli1992, Caffarelli1996} for the Monge--Ampère equation.

A recurring theme in the modern theory of PDE is to study well-posedness in spaces of weak solutions. 
In this paper we are interested in studying Sobolev solutions of \eqref{eq:jac} when $f\in L^p(\Omega)$ and we will focus, in particular, on the case where the solutions are in a critical or supercritical Sobolev space. In contrast with the rich classical theory, almost nothing is known in this setting. Indeed, and since the Jacobian is an $n$-homogeneous polynomial in $\D u$, we have the following natural question:

\begin{question}\label{question:main}
Let $1\leq p<\infty$ and assume $f$ satisfies (\ref{eq:hypdata}). 
\begin{enumerate}
\item\label{it:p>1} If $p>1$, $f\in L^p(\Omega)$, does (\ref{eq:jac}) have a solution $u\in W^{1,np}(\Omega,\R^n)$? See \cite[Question \ 3]{Ye1994}.
\item\label{it:p=1} If $p=1$, $f\in L\log L(\Omega)$, does (\ref{eq:jac}) have a solution $u\in W^{1,n}(\Omega,\R^n)$? See \cite[page 206]{Hogan2000}.
\end{enumerate}
\end{question}

The case $p=\infty$ is ruled out from Question \ref{question:main} and it was addressed in \cite{Burago1998, McMullen1998}, see also \cite{Viera2018}, where it was shown that in this case the answer is negative, even when $f$ is continuous. Partial positive results were obtained in \cite{Riviere1996}.
The case $p=1$, where one must replace $L^1(\Omega)$ with $L\log L(\Omega)$ \cite{Hogan2000}, also deserves special attention. Indeed, the Jacobian benefits from an improved integrability, first noticed by \textsc{M\"uller} \cite{Muller1990} and then generalised in \cite{Coifman1993}; an estimate up to the boundary was obtained in \cite{Hogan2000} and, in the planar case, a beautiful sharp version can be found in \cite{Astala2012}. As noticed in \cite{Guerra2019}, this improved integrability is essentially a by-product of the weak continuity of the Jacobian, see Section \ref{sec:improvedintegrability} for more details. 
We note that the special behaviour of the Jacobian at the endpoints $p=1$, $p=\infty$ is not unexpected:
for the divergence equation, which is a linearisation of the prescribed Jacobian equation, one also has non-existence of solutions when $f\in L^1$ or $f\in L^\infty$, due to the failure of estimates in these cases \cite{Mityagin1958, Ornstein1962}, see also \cite{Faraco2020, Kirchheim2016} and \cite[§2]{Bourgain2002}.

The answer to both parts of Question \ref{question:main} is negative in a rather strong sense. 
Indeed, let us  consider for $1\leq p<\infty$ the complete metric spaces
\begin{equation*}
\begin{split}
X_p & \equiv \left\{f\in 
L^p(\Omega): f \tp{ satisfies } \eqref{eq:hypdata}
\right\}, \qquad \tp{for } p>1\\
X_1 & \equiv \left\{f\in 
 L\log L(\Omega): f \tp{ satisfies } \eqref{eq:hypdata}
\right\}.
\end{split}
\end{equation*}
We then have:

\begin{bigthm*}{Main Theorem}
\label{thm:main}
Fix $0<c<1$ and $1\leq p<\infty$.

\noindent There is $f\in X_p$ such that the Dirichlet problem \eqref{eq:jac} has no solution $u\in W^{1,np}(\Omega,\R^n)$.
In fact, for a Baire-generic $f\in X_p$, \eqref{eq:jac} has no solution in the space $\bigcup_{n\leq q,\, p<q} W^{1,q}(\Omega,\R^n).$
\end{bigthm*}

We briefly describe the proof of the main theorem. Assume for simplicity that $\Omega=B$ is a ball.
The starting point of our proof is the observation that if  $f\in L^p(B)$ is a radially symmetric function then the unique spherically symmetric solution of \eqref{eq:jac} is, in general, in $W^{1,p}\exc W^{1,q}$ for any $q>p$, see Proposition \ref{prop:ye}. Hence we first show that any $L^p$ function admits perturbations which are radially symmetric in a small neighbourhood of $\p B$ and such that the $W^{1,q}$-norm of the symmetric solutions in that neighbourhood is unbounded. The crucial step of the proof is to show that, for these perturbations, the symmetric solutions have comparable $q$-Dirichlet energy to that of the energy-minimiser: hence we deduce that the $W^{1,q}$-norm of $q$-energy-minimal solutions (see p. 8) is unbounded at every $f\in X_p$.
The weak continuity of the Jacobian, together with an application of the  Baire Category  Theorem, then concludes the proof.

It will be clear from the proof that the boundary condition $u=\tp{id}$ on $\p \Omega$ in \eqref{eq:jac} can be relaxed to the requirement that $u(x) \in \p \Omega$ for $\mathscr H^{n-1}$-a.e.\ $x\in \p \Omega$.
Under this generalised boundary condition, our arguments also deal easily with the case where the compatibility condition in \eqref{eq:hypdata} is replaced by the condition $\frac{1}{|\Omega|}\int_\Omega f(x)\d x=k$ for some fixed $k\in \N$. In general, the number $k$ gives an essential upper bound on the multiplicity function of solutions.

Our main theorem raises the following question, which we do not address here:
\begin{question}\label{qu:open}
For $n\leq p=q<\infty$ and $f\in X_p$, does (\ref{eq:jac}) have a solution $u\in W^{1,q}(\Omega,\R^n)$?
\end{question}

Question \ref{qu:open} is natural from the viewpoint of the regularity theory for optimal transport maps. In Optimal Transportation, the underdetermined equation \eqref{eq:jacPDE} is usually turned into a determined elliptic equation by imposing the constraint $u=\D \phi$ for some convex scalar function $\phi$: thus $\phi$ satisfies the Monge--Ampère equation $\det \D^2 \phi=f$, see \cite{Brenier1991, DePhilippis2014,Trudinger2009}.  The relation between the prescribed Jacobian equation and the Monge--Ampère equation is analogous to that of their linear counterparts, the divergence equation and the Poisson equation, see Table \ref{table}. Moreover, the estimates for the Monge--Ampère equation scale as in these linear problems, which is precisely the scaling in Question \ref{qu:open}.

However, the failure of a certain estimate for the determined problem does not imply that the same happens for the underdetermined problem, as shown by \textsc{Bourgain}--\textsc{Brezis} \cite{Bourgain2002}. Surprisingly, even in the linear setting,
\textit{solutions of the determined problem do not have optimal regularity for the underdetermined problem}! This is also the case for the non-linear problem \eqref{eq:jacPDE}, at least if we only require $\inf_\Omega f\geq 0$. Indeed, at \cite[page 293]{Ye1994}, Ye gives an example in $\Omega = B \subset \R^2$ such that \eqref{eq:jac} has a smooth solution but the corresponding solution of the Monge--Amp\`{e}re equation is only $C^{1,1}$ regular.

\begin{table}[h]
\centering
\begin{tabular}{c|c|c}
& Linear & Non-linear \\\hline
Determined &  $\triangle \phi = f$ & $ \det \D^2 \phi = f$ \\\hline
Underdetermined  & $\tp{div}\, u = f$ & $\det \D u = f$
\end{tabular}
    \captionsetup{width=.8\linewidth}
\caption{Particular solutions to the underdetermined problem are obtained by taking $u=\D \phi$ for $\phi$ a solution of the corresponding determined problem.}
\label{table}
\end{table}

In \cite{Koumatos2015} it was shown that for any $f\in L^1(\Omega)$ and any $q<n$ there is a \textit{pointwise solution} $u\in W^{1,q}(\Omega,\R^n)$ of \eqref{eq:jac}. Thus the range of exponents in Question \ref{qu:open} is, in some sense, optimal.
However, it is perhaps more natural to look for distributional solutions of \eqref{eq:jac} once $q<n$. Indeed, for $\frac{n^2}{n+1}\leq q < n$ there is a well-defined distributional Jacobian which is, in general, different from the pointwise Jacobian. Moreover, we recall that for $n\leq q$, as in this paper, the distributional Jacobian and the pointwise Jacobian agree. We do not discuss distributional Jacobians further, instead referring the reader to \cite{Alberti2003, Ball1977, Brezis2011b, Muller1990a}.

To conclude the introduction we briefly discuss the motivation for the work initiated on this paper, which is twofold.
Firstly, the Jacobian plays an important role in Continuum Mechanics, specially in nonlinear elasticity, see e.g.\ \cite{Ball1981a, Dacorogna1981, Fischer2019, Hogan2000}. This is not surprising since, for a Sobolev map, positivity of the Jacobian implies a weak form of local invertibility \cite{Fonseca1995}.  Secondly, in $\R^n$, there is an intimate connection between the Jacobian equation, commutator estimates and factorization theorems in Hardy and $L^p$ spaces \cite{Hytonen2018,Hytonen2019,Lindberg2020}; an outstanding open problem in the area is to determine whether the Jacobian is onto the Hardy or $L^p$ space \cite{Coifman1993,Iwaniec1997,Lindberg2017}.
Positive partial results can be found in \cite{Coifman1993,Hytonen2019,Lindberg2015,Lindberg2020} and negative ones in \cite{Lindberg2017}. We refer the reader to our forthcoming work \cite{GuerraKochLindberg2020b} for further details and progress towards the possible non-surjectivity of the Jacobian.

\section{Notation and background results}\label{sec:prelims}

In this section we gather a few preliminary results for the convenience of the reader.

The set $\Omega$ will always denote a bounded smooth domain in $\R^n$. Moreover, in Section \ref{sec:proof}, we will further assume that $\Omega$ is \textit{uniformly convex}, in the sense that the second fundamental form of $\p \Omega$ is uniformly positive.

We will often write $\J u\equiv \det \D u$. As is customary, the symbols $a \approx b$ and $a\lesssim b$ are taken to mean that there is some constant $C>0$ independent of $a$ and $b$ such that $C^{-1} a \leq b \leq C a$ and $a\leq C b$, respectively. We denote by $\mb S^{n-1}$ the unit sphere in $\R^n$ and, for $r<R$, we write 
\begin{align*}
\mb A(r,R)\equiv \{x\in \R^n: r<|x|<R\}, \qquad
\mb A[r,R]\equiv \{x\in \R^n: r\leq |x|\leq R\}.
\end{align*} 
Here $|\cdot |$ denotes the Euclidean norm of a vector in $\R^n$; for a matrix $A\in \R^{n\times n}$, we also write $|A|$ for its operator norm.
We represent by $\mathscr L^n$ the $n$-dimensional Lebesgue measure, by $\mathscr H^{n-1}$ the $(n-1)$-dimensional Hausdorff measure and we write
$\omega_n \equiv \mathscr L^n(B_1(0))$.

\subsection{Mappings of finite distortion}
\label{sec:mappingsofFD}

In this section we gather a few results from the theory of mappings of finite distortion, which is detailed in \cite{Hencl2014a, Iwaniec2001}, see also \cite{Astala2009} for the planar case.

\begin{ndef}
Let $u\in W^{1,1}_\tp{loc}(\Omega,\R^n)$ be such that $0 \leq \J u\in L^1_\tp{loc}(\Omega)$. We say that $u$ is a \emph{map of finite distortion} if there is a function $K\colon \Omega\to [1,\infty]$  such that  $K<\infty$ a.e.\ in $\Omega$ and
$$|\D u(x)|^n \leq K(x)\, \J u(x) \quad \tp{ for a.e.\ } x \tp{ in } \Omega.$$
If $u$ has finite distortion, we can set $\tp{K}u(x) = \frac{|\D u|^n}{\J u(x)}$ if $\J u(x)\neq 0$ and $\tp{K}u(x)=1$ otherwise; this function is the (optimal) \emph{distortion} of $u$.
\end{ndef}

We state some basic properties of the topological degree, referring the reader to \cite{Fonseca1995} for further details.

\begin{lemma}\label{lemma:degree}
For $u\in C^0(\overline \Omega, \R^n)$ and $x\in\R^n\exc u(\p\Omega)$, the \emph{topological degree} of $u$ at $x$ with respect to $\Omega$, denoted by $\tp{deg}(x,u,\Omega)$, has the following properties:
\begin{enumerate}
\item\label{it:integer} $\tp{deg}(\cdot,u,\Omega)\colon \R^n\exc u(\p\Omega)\to \Z$;
\item\label{it:boundary} 
if $u=v$ on $\p \Omega$, then $\tp{deg}(\cdot, u,\Omega)=\tp{deg}(\cdot, v,\Omega)$ on $\R^n\exc u(\p\Omega)$;
\item\label{it:inclusion}
if $x\not \in u(\p \Omega)$ and $\tp{deg}(x,u,\Omega)\neq 0$ then $x\in u(\Omega)$.
\end{enumerate}
\end{lemma}

For our purposes, the following classical result \cite{Vodopyanov1976} is particularly relevant:

\begin{thm}\label{thm:mappingsofFD}
Let $u\in W^{1,n}(\Omega,\R^n)$ be a map of finite distortion. Then $u$ has a continuous representative with the Lusin (N) property.
\end{thm}

Recall that a map has the \textit{Lusin (N) property} if it maps $\mathscr L^n$-null sets to $\mathscr L^n$-null sets.
Theorem \ref{thm:mappingsofFD} gives no improvement in the supercritical case, since the continuous representative of a map $u\in W^{1,p}$, $p>n$, has the Lusin (N) property. 

Theorem \ref{thm:mappingsofFD} shows that solutions of \eqref{eq:jac} satisfy the change of variables formula:

\begin{cor}\label{cor:changeofvars}
Let $u\in W^{1,n}(\Omega,\R^n)$ be a solution of \eqref{eq:jac} and \eqref{eq:hypdata}. Then
$$\mathscr L^n(u(E)) = \int_E f \d x\qquad \tp{for all measurable sets } E\subseteq \Omega.$$
\end{cor}

\begin{proof}
The non-degeneracy assumption \eqref{eq:hypdata} implies that $u$ is a map of finite, and even integrable, distortion. Hence, by Theorem \ref{thm:mappingsofFD}, the coarea formula applies, see for instance \cite[Theorem 5.23]{Fonseca1995}: if 
$\mc N$ is
the multiplicity function $\mc N(y,u,E)\equiv \#\{x\in E: u(x)=y\}$, then
\begin{equation}\int_{u(E)} \mc N(y,u,E) \d y= \int_E f\d x, \qquad
\tp{for all measurable sets } E\subseteq \Omega
\label{eq:area}.
\end{equation}
 Moreover, $u$ is injective a.e., that is $\mc N(y,u,\Omega)\leq 1$ for a.e.\ $y\in \R^n$: it suffices to apply \eqref{eq:area} with $E=\Omega$, recalling that $\int_\Omega f\d x=|\Omega|$ by \eqref{eq:hypdata}. Hence the conclusion follows from \eqref{eq:area}.
\end{proof}

Although this will not be needed for the proof of the Main Theorem, it is worthwhile mentioning that solutions of \eqref{eq:jac} in a sufficiently good Sobolev space are homeomorphisms. The sharp statement follows for instance from the main result in \cite{Hencl2002}:

\begin{thm}\label{thm:HenclMaly}
Let $u\in W^{1,n}(\Omega,\R^n)$ be a continuous map with distortion in $L^{n-1}(\Omega)$. If there is a compact set $K\subset \Omega$ such that $u|_{\Omega\exc K}$ is discrete then $u$ is open and discrete.
\end{thm}

Recall that a map is \emph{discrete} if the preimage of every point is locally finite. A counterexample due to \textsc{Ball} \cite{Ball1981a} shows that the conclusion of Theorem \ref{thm:HenclMaly} does not hold if the distortion is in $L^q(\Omega)$ for $q<n-1$. We then deduce:

\begin{cor}\label{cor:solsarehomeos}
Let $u\in W^{1,n(n-1)}(\Omega, \R^n)$ be a solution of \eqref{eq:jac},  \eqref{eq:hypdata}. Then
 $u\in \tp{Hom}(\overline \Omega, \overline \Omega)$.
\end{cor}

\begin{proof}
It follows from our assumptions that the distortion of $u$ is in $L^{n-1}(\Omega)$.  Consider a sufficiently small neighbourhood $\Omega_\delta$ of $\Omega$ and extend $u$ to be the identity in $\Omega_\delta\exc \Omega$. Theorem \ref{thm:HenclMaly}, applied in $\Omega_\delta$, shows that $u|_\Omega$ is open and discrete and it follows that $u\in \tp{Hom}(\overline \Omega,u(\overline \Omega))$, see e.g.\ \cite[Theorem 3.27]{Hencl2014a}. Moreover, $u(\overline \Omega)=\overline \Omega$. Indeed, it  suffices to prove that $u(\Omega)\subseteq \Omega$, since 
\begin{equation}
u(\overline \Omega)=u(\Omega) \sqcup  u(\p \Omega)= u(\Omega)\sqcup \p \Omega
\label{eq:homeospreservebdry}
\end{equation} 
where the unions are disjoint, and $\Omega\subseteq u(\Omega)$ by Lemma \ref{lemma:degree}.
Suppose that there is $x\in \Omega$ such that $u(x)\in \R^n \exc \overline \Omega$. Take $y\in \Omega$ such that $u(y)\in \Omega$ and consider a continuous path in $\Omega$ joining $x$ and $y$. The image of such a path under $u$ must cross $\p \Omega$ somewhere, which contradicts \eqref{eq:homeospreservebdry}.
\end{proof}

In fact, a simple argument using the change of variables formula shows that the inverse map $u^{-1}$ is in $W^{1,n}(\Omega,\Omega)$. Note that the situation in the planar case is particularly pleasant, since then $n-1=1$.
Other results in the direction of Corollary \ref{cor:solsarehomeos} can be found in \cite{Ball1981a, Sverak1988}.

\subsection{Improved integrability of the Jacobian}
\label{sec:improvedintegrability}

In this section we recall standard facts concerning the improved integrability of the Jacobian. This phenomenon was first noticed by \textsc{M\"uller} in \cite{Muller1990} and then generalised by \textsc{Coifman}--\textsc{Lions}--\textsc{Meyer}--\textsc{Semmes} in \cite{Coifman1993}. The latter paper shows that the Jacobian determinant of maps in ${\dot W}^{1,n}(\R^n,\R^n)$ lies in the Hardy space $\mathscr H^1(\R^n)$, a fact which was extended in \cite{Guerra2019} to a large class of polynomials and differential operators.
These results are essentially local and, as we are interested in the behaviour of Jacobians on domains, we will need the following global version from \cite{Hogan2000}:

\begin{thm}\label{thm:LlogL}
Let $u\in W^{1,n}(\Omega,\R^n)$ be such that $\J u\geq 0$ a.e.\ in $\Omega$. Then, for any $p>n$,
$$\Vert \J u \Vert_{L\log L(\Omega)}\leq C(\Vert u \Vert_{W^{1-1/p,p}(\p \Omega)}+\Vert \D u \Vert_{L^n(\Omega)}).$$
\end{thm}

There is a more general version of Theorem \ref{thm:LlogL} when $\J u$ changes sign, but we will not need it here. The space $L\log L(\Omega)$ in the theorem is a particular case of an Orlicz--Zygmund space and it can be defined as the space of those measurable functions $f\colon \Omega\to \R$ such that
\begin{equation}
\Vert f \Vert_{L\log L(\Omega)}\equiv 
\int_\Omega | f(x) | \log \bigg(e+\frac{|f(x)|}{\Vert f \Vert_{L^1}}\bigg) \d x<\infty.
\label{eq:LlogL}
\end{equation}
The above expression defines a complete norm, although it is not immediate that it satisfies the triangle inequality, see \cite[§8]{Iwaniec1999b}. The standard Luxemburg norm in $L\log L(\Omega)$ is different but equivalent to the one in \eqref{eq:LlogL}; we chose the latter since it is easier to use in calculations.
We refer the reader to \cite{Stein1969} for the relation between the space $L\log L$ and local Hardy spaces.

\section{Radial data and radial stretchings}

In this section we give a description of the regularity of radial stretchings solving \eqref{eq:jac}. 

A function $f\colon B_R(0)\to \R$ is \textit{radially symmetric} if $|x|=|y|\implies f(x)=f(y)$ and we identify any such function with a function $f\colon [0,+\infty)\to \R$ in the obvious way.
Under natural assumptions, and when the data is radially symmetric, equation (\ref{eq:jac}) admits a unique spherically symmetric solution, that is, 
a solution of the form
$u(x)=\rho(|x|) \frac{x}{|x|}$; we refer to these maps as \textit{radial stretchings}. More precisely, c.f.\ \cite[Lemma 4.1]{Ball1982}, we have:

\begin{lemma}\label{lemma:ball}
Let $n>1$ and $1\leq p<\infty$. The map $u(x)=\rho(|x|) \frac{x}{|x|}$ is in $W^{1,p}(B_R(0), \R^n)$  if and only if $\rho$ is absolutely continuous on $(0,1)$ and 
$$\Vert \D u \Vert_{L^p(B_R(0))}^p\approx_n \int_0^R \left(|\dot \rho(r)|^p + \left|\frac{\rho(r)}{r}\right|^p \right)r^{n-1} \d r<\infty;$$
a similar statement holds for $p=\infty$.
In this case, for a.e.\ $x$ in $B_R(0)$, and writing $r=|x|$, we have the formulae
\begin{gather}
\label{eq:jacobianradialstretching}
\J u(x)= \frac{1}{{r^{n-1}}}\dot \rho(r)\rho^{n-1}(r),\\
\label{eq:derivative}
\D u(x) = \frac{\rho(r)}{r} \tp{Id} + \left(\dot \rho(r)-\frac{\rho(r)}{r}\right)
\frac{x\otimes x}{r^2}.
\end{gather}
In particular, if $u$ solves (\ref{eq:jac}), then 
\begin{equation}
\label{eq:rho}
\rho^n(r)=\int_0^r n f(s) s^{n-1} \d s.
\end{equation}
\end{lemma}

We also record here the standard notation $\p_r u(x) \equiv (\D u(x))\cdot \frac{x}{r}$; in particular, if $u$ is a radial stretching as in Lemma \ref{lemma:ball}, then
\begin{equation}
\label{eq:radialderivative}
\p_r u(x) = \dot \rho(r) \frac{x}{r}.
\end{equation}

The following proposition follows easily from Lemma \ref{lemma:ball}, see also \cite[Theorem 6]{Ye1994}.

\begin{prop}\label{prop:ye}
Let $\Omega=B_R(0)$ and $p\in [1, +\infty]$. If $f\in L^p(\Omega)$ is radially symmetric and satisfies \eqref{eq:hypdata} then the unique radial stretching $u$ solving \eqref{eq:jac} is in $W^{1,p}(B_R(0))$ and 
\begin{equation}
\Vert \D u \Vert_{L^{p}(B_R(0))} \lesssim_n \frac{\Vert f \Vert_{L^p(B_R(0))}}{(\inf f)^\frac{n-1}{ n}} + R^{\frac{n-1}{p}}\Vert f \Vert_{L^p(B_R(0))}^{\frac 1 n}.
\label{eq:radialestimate}
\end{equation}Moreover, this inequality is sharp: in general $u\not \in W^{1,q}$ for any $q>p$.
\end{prop}

\begin{proof}
We assume that $1\leq p<\infty$, as the case $p=\infty$ is similar. Using \eqref{eq:rho} and Jensen's inequality, we deduce that
$$\left\vert \frac{\rho(r)}{r}\right\vert^p\lesssim \left(\fint_{B_r(0)} f(x) \d x\right)^\frac p n \leq \frac 1 r \left(\int_{B_r(0)} f(x)^p \d x \right)^\frac 1 n\leq \frac 1 r \,\Vert f \Vert_{L^p(B_R(0))}^\frac p n
$$
and therefore, integrating in $r$,
$$
\int_0^R \left\vert \frac{\rho(r)}{r}\right\vert^p r^{n-1} \d r 
\lesssim R^{n-1} \Vert f \Vert_{L^p(B_R(0))}^\frac p n.
$$
Since $f$ satisfies \eqref{eq:hypdata}, we deduce that $(\inf f)^{\frac 1 n}\leq \rho(r)/r$. Thus, from \eqref{eq:jacobianradialstretching}, we have
$$\dot \rho(r) = \frac{r^{n-1}}{\rho^{n-1}(r)} f(r)
\quad
\implies
\quad
\int_0^R |\dot \rho(r)|^p r^{n-1} \d r \lesssim \frac{\Vert f\Vert_{L^p(B_R(0))}^p}{(\inf f)^{\frac{p(n-1)}{ n}}}
$$
and the desired estimate follows from Lemma \ref{lemma:ball}. To see that this inequality is sharp, it suffices to note that for $\delta>0$ there is a constant $C_\delta$ such that $C_\delta f(r) \leq \dot \rho(r)$ for $r\in (\delta,R)$. By choosing $f\in L^p(0,R)$ such that $f\not\in L^{q}(\delta,R)$ for $q>p$ we have $\D u\not \in L^{q}(B_R(0))$.
\end{proof}

\section{Non-solvability of the prescribed Jacobian equation}\label{sec:proof}

This section is dedicated to the proof of the main result. We begin by taking $\Omega=B_1(0)$.

For $p\in [1,+\infty)$ and $\eta\in [0,\frac{c-1}{c})$, let
\begin{equation}
\begin{split}
Z_p^{\eta} & \equiv \left\{f\in L^p(B_1): \fint_{B_1} f \d x = 1, \, f\geq (1-\eta)c\right\}, \qquad \tp{for } p>1\\
Z_1^{\eta} & \equiv \left\{f\in L\log L(B_1): \fint_{B_1} f \d x = 1, \, f\geq (1-\eta)c\right\};
\label{eq:defZeta}
\end{split}
\end{equation}
and write $Z_p\equiv Z^0_p$. We make a few immediate remarks about the sets $Z_p^\eta$:
\begin{enumerate}
\item clearly $Z_p^{\eta_1}\subset Z_p^{\eta_2}$ if $\eta_1<\eta_2$;
\item each set $Z_p^\eta$ is a complete metric space under the distance
$$\tp{dist}_p(f,g)\equiv 
\begin{cases}
\Vert f -g \Vert_{L^p(B_1)}, & p>1\\
\Vert f -g \Vert_{L\log L(B_1)}, & p=1
\end{cases};$$
\item if $c=1$ then the only elements of $Z_p$ are functions which are 1 a.e.\ in $B_1$. The same holds, more generally, if $\eta= \frac{c-1}{c}$. Thus we assume throughout this section, without loss of generality, that $c<1$ and $\eta<\frac{c-1}{c}$.
\end{enumerate}
We write $B_{Z^\eta_p}(f,\e)$ for a ball of radius $\e$  around $f$, under this distance.
Given $f\in Z_p$ and $\max \{p,n\}\leq q$, we consider the $q$-\textit{energy of the datum} $f$, defined by
$$\mc E_q(f)\equiv 
\inf\left\{\int_{B_1}|\D v|^{q}\d x : v \tp{ solves } \eqref{eq:jac}\right\}.$$
Thus \eqref{eq:jac} admits a $W^{1,q}$ solution if and only if $\mc E_q(f)<\infty$. We say that $w$ is a $q$-\textit{energy-minimal solution} for $f$ if $w$ solves \eqref{eq:jac} and moreover
$$\mc E_q(f) = \int_{B_1(0)} |\D w|^q \d x.$$

We begin with the following lemma:

\begin{lemma}\label{lemma:perturbation}
Fix $1\leq p<\infty$ and let $\e,\eta>0$ be arbitrary. For any $f\in Z_p$
there are sequences $f_j\in B_{Z^\eta_p}(f,\e)$ and $R_j\nearrow 1$ such that 
\begin{enumerate}
\item\label{it:perturbation} $\tp{dist}_p(f_j,f)\to 0$ as $j\to \infty$;
\item\label{it:symmetry} $f_j$ is radially symmetric in $\mb A(R_j,1)$;
\item\label{it:radialblowup}  if $u_j$ is the radial stretching such that $\J u_j=f_j$ in $\mb A(R_j,1)$ and $u_j=\tp{id}$ on $\mb S^{n-1}$, then
$$\lim_{j\to \infty} \int_{\mb A(R_j,1)} |\p_r u_j(x)|^q\d x =  +\infty$$
for any $q>p$.
\end{enumerate}
\end{lemma}

\begin{proof}
For $\gamma, R \in (\frac 3 4,1)$, consider the functions
\begin{equation}
\label{eq:deffgammaR}
f_{\gamma,R}(x) \equiv \begin{cases}
\frac{\gamma^n}{\fint_{B_R} f \d x} f(x), & 0 < |x| < R, \\
\frac{ M [\gamma R + M (|x|-R)]^{n-1}}{|x|^{n-1}}, & R \le |x| < 1,
\end{cases}
\qquad \tp{where } M \equiv \frac{1-\gamma R}{1-R}.
\end{equation}
The choice of $M$ ensures that $\fint_{B_1} f_{\gamma, R}\d x = 1$, since
\begin{equation}
\label{eq:intfgammaR}
\int_{R}^1 f_{\gamma,R}(r) r^{n-1} \d r = \frac{1-(\gamma R)^n}{n}.
\end{equation}
Fix $\alpha\in (-\frac{q}{p},-1)$ and choose $\gamma=\gamma(R)$ in such a way that
\begin{equation}
\label{eq:choicegamma}
1-\gamma R  = (1-R)^{1+\alpha/q};
\end{equation}
in particular, we have $\gamma(R)\to 1$ as $R\to 1$.
Take any sequence $R_j \nearrow 1$ and consider the associated numbers $\gamma_j=\gamma_j(R_j)$. We will prove that, for $j$ large enough, $f_j\equiv f_{\gamma_j,R_j}$ satisfies the properties above. 

Concerning the lower bounds of the sequence, for $r\in (R,1)$, we have the simple estimate
$$f_{\gamma,R}(r) = M \left(\frac{\gamma R + M(r-R)}{r}\right)^{n-1} \geq 1\times \gamma^{n-1};$$ 
thus $f_{\gamma,R}(r)>c$ for $\gamma$ sufficiently close to 1.
For $r\in (0,R)$, clearly $f_{\gamma,R}(r)\geq \gamma^n c/\fint_{B_R} f\to c$ as $\gamma,R\to 1$. Hence the lower bounds are satisfied.
Moreover, \ref{it:symmetry} clearly holds. 

For \ref{it:radialblowup},
denote by $u_{\gamma,R}(x) = \rho_{\gamma,R}(r) \frac{x}{r}$ the unique radial stretching solving $\J u_{\gamma,R} = f_{\gamma,R}$ in $\mb A(R,1)$ and such that $u_{\gamma,R}=\tp{id}$ on $\mb S^{n-1}$. By Lemma \ref{lemma:ball} we find that
\begin{equation} \label{eq:radialsol}
\rho_{\gamma,R}(r) =\gamma R + M (r-R),
\qquad r\in (R,1]
\end{equation}
and so, by \eqref{eq:radialderivative}, since $r^{n-1} \d r\approx \d r$ for $r\in (\frac 1 2, 1)$,
\[\int_{\mathbb{A}(R,1)} |\partial_r u_{\gamma,R}(x)|^{q} \d x\approx 
\int_R^1 |\dot \rho(r)|^q \d r = 
M^{q} (1-R)= \frac{(1-\gamma R)^{q}}{(1-R)^{q-1}}= (1-R)^{1+\alpha}.\]
Thus, since $\alpha <-1$, we see that \ref{it:radialblowup} also holds.

Hence it remains to prove \ref{it:perturbation}, and we split the proof into two cases.

\textbf{Case $\bm{p>1}$:}
Whenever $\gamma,R \nearrow 1$, we have
$$\int_{B_R} |f-f_{\gamma,R}|^p \d x = \left| \frac{\gamma^n}{\fint_{B_R} f \d x} - 1 \right|^p \int_{B_R} |f|^p \d x \to 0.$$
Thus, since $\int_{\mb A(R,1)} |f(x)|^p\d x\to 0$ as $R\to 1$, in order to prove that $f_{\gamma,R}\to f$ in $L^p(B_1)$ it suffices to show that, as $\gamma,R\nearrow 1$, 
\begin{equation}
\label{eq:convergencegoal}
\int_{\mb A(R,1)} |f_{\gamma,R}(x)|^p \d x\to 0.
\end{equation}
For this, note that for $|x|=r\in (R,1) $, since $\gamma R<1$ and $R>\frac 1 2$,
$$f_{\gamma,R}(x) \leq \frac{M}{r^{n-1}}
\leq 2^{n-1} \frac{(1-\gamma R)}{1-R}.$$
Thus, as $r^{n-1} \d r \approx \d r$ for $r\in (\frac 1 2, 1)$,
$$\int_{\mb A(R,1)} |f_{\gamma,R}(x)|^p \d x
\approx \int_R^1 |f_{\gamma,R}(r)|^p \d r \lesssim \frac{(1-\gamma R)^p}{(1-R)^{p-1}}= (1-R)^{\alpha\frac p q  +1}.$$
Since $\a>-\frac q p$ we have that $\a \frac p q + 1>0$ and so \eqref{eq:convergencegoal} is proved.

\textbf{Case \bm{$p=1$}:}
As before, we have that, as $\gamma,R\nearrow 1$,
\begin{align*}
\|f_{\gamma,R}-f\|_{L\log L(B_R)} = \left|1-\frac{\gamma^n}{\fint_{B_R} f\d x}\right|\|f\|_{L\log L(B_R)}\to 0
\end{align*}
and also $\|f \|_{L\log L(\mb A(R,1))} \to 0$ as $R\to 1$. Thus it suffices to show $\|f_{\gamma,R}\|_{L\log L(\mb A(R,1)}\to 0$ as $R\to 1$. 
We make the simple observation that
$$1-(\gamma R)^n =(1-\gamma R)(1+\gamma R + \dots + (\gamma R)^{n-1})\approx_n (1-\gamma R).$$
Then, using the estimates on $f_{\gamma,R}$ from the previous case and \eqref{eq:intfgammaR}, we obtain
\begin{align*}
\|f_{\gamma,R}\|_{L\log L(\mb A(R,1))} & 
\approx \int_{\mb A(R,1)} f_{\gamma,R}(x) \log\left(e+\frac{f_{\gamma,R}(x)}{\|f_{\gamma,R}\|_{L^1(\mb A(R,1))}}\right)\d x\\
& \lesssim \int_R^1 \frac{1-\gamma R}{1-R} \log \left(e+\frac{1-\gamma R}{(1-R)(1-\gamma^n R^n)}\right)\d r\\
& \leq (1-R)^{\alpha/q+1}\log\left(e+\frac{1}{1-R}\right).
\end{align*}
The right-hand side converges to zero as $R\to 1$, since $\alpha/q+1>0$.
\end{proof}

Our goal is to show that  $\mc E_q(f_{\gamma,R})\to +\infty$ as $\gamma,R\to 1$. The idea is that energy-minimal solutions are controlled by the radial solution. 

\begin{prop}\label{prop:quasiminimizers}
Fix $1\leq p<q$ and suppose that $v\in W^{1,n}(B_1,\R^n)$ is a solution of $\tp{J} v = f_{\gamma,R}$ with $v=\tp{id}$ on $\p B_1$, where $f_{\gamma,R}$ is as in \eqref{eq:deffgammaR} and \eqref{eq:choicegamma}.  There is a constant $C=C(n,q)>0$ such that
$$\int_{\mb A(R,1)} |\p_r u(x)|^{q}\d x \leq 
C\int_{\mb A(R,1)} |\p_{r} v(x)|^{q}\d x,
$$ where $u=u_{\gamma,R}$ is the radial stretching such that $\J u_{\gamma,R}=f_{\gamma,R}$ in $\mb A(R,1)$ and $u=\tp{id}$ on $\mb S^{n-1}$.
\end{prop}

\begin{proof}
Throughout the proof $\theta$ will denote an element of $\mb S^{n-1}$ and we write $r \theta$ for the corresponding element in a sphere of radius $r$.

Since $R > \frac 1 2$, from H\"{o}lder's inequality,
\begin{align*}
2^{n-1}\fint_{\mb S^{n-1}} \fint_R^1 |\p_r v(r\theta)|^{q} r^{n-1} \d r \d \theta
& \ge \fint_{\mb S^{n-1}} \fint_R^1 |\p_r v(r \theta)|^{q} \d r \d \theta \\
&\ge \left( \fint_{\mb S^{n-1}} \fint_R^1 |\p_r v(r \theta)| \d r \d \theta \right)^{q}.
\end{align*}
On the other hand, since $\p_r u(r \theta) = M \theta$ for all $r \in (R,1)$, c.f.\ \eqref{eq:radialsol}, we can estimate
\begin{align*}
\left( \fint_{\mb S^{n-1}} \fint_R^1 |\p_r u(r \theta)| \d r \d \theta \right)^{q}
= \fint_{\mb S^{n-1}} \fint_R^1 |\p_r u(r \theta)|^{q} \d r \d \theta 
\ge \fint_{\mb S^{n-1}} \fint_R^1 |\p_r u(r \theta)|^{q} r^{n-1} \d r \d \theta.
\end{align*}
Thus, the proposition will be proved once we show that
\begin{equation*} 
\label{eq:final boss}
\int_{\mb S^{n-1}} \int_R^1 |\p_r v(r \theta)| \d r \d \theta \ge  
\frac{1}{4n^2}  \int_{\mb S^{n-1}} \int_R^1 |\p_r u(r \theta)| \d r \d \theta
\end{equation*}
for all $R \in (\frac 1 2,1)$. Note that $\int_R^1 |\p_r u(r \theta)| \d r = (1-R) M = 1 - \gamma R$ for all $\theta \in\mb S^{n-1}$.

Let us take the set $$\Theta\equiv\{\theta\in\mb S^{n-1}: r\mapsto u(r\theta) \tp{ is absolutely continuous in } [R,1]\}.$$
For $\lambda\in(0,1)$ to be chosen later, let
\begin{align*}
\Theta_1 \equiv\left\{\theta \in \Theta: \int_R^1 |\p_r v(s\theta)| \d s \leq \lambda\right\},\qquad
\Theta_2 \equiv\left\{\theta \in \Theta: \int_R^1 |\p_r v(s\theta)| \d s > \lambda\right\}.
\end{align*}
By the ACL property of Sobolev functions, $\mathscr H^{n-1}(\Theta)=\mathscr H^{n-1}(\mb S^{n-1})$ and so the set  $\mb S^{n-1}\exc \Theta$ is an $\mathscr H^{n-1}$-null set. Fubini's theorem and the fact that $v$ satisfies the Lusin (N) property, c.f.\ Theorem \ref{thm:mappingsofFD}, then implies
$$\mathscr L^n\left(v(\Theta_1\times [R,1])\right)+ \mathscr L^n\left(v(\Theta_2\times [R,1])\right)
= \mathscr L^n (v(\mb A(R,1)).
$$

By the fundamental theorem of calculus, for $\theta \in \Theta_1$, $r\in [R,1]$,
$$1-|v(r \theta)|\leq |v(\theta) - v(r \theta)| \leq \int_r^1 |\p_r v(s \theta)| \d s \leq \lambda$$
and thus $v(\Theta_1 \times [R,1])\subset \mb A[1-\lambda,1]$. Combined with the change of variables formula from Corollary \ref{cor:changeofvars}, \eqref{eq:intfgammaR} and Bernoulli's inequality, this estimate yields
\begin{align*}
\mathscr L^n(v(\Theta_2\times [R,1])) & = \int_{\mb A(R,1)} f_{\gamma,R}(x) \d x -\mathscr L^n(v(\Theta_1\times [R,1]))
\\ &\geq \omega_n \left[(1-\gamma^n R^n)-(1-(1-\lambda)^n)\right]\\
& \geq \omega_n \left(1-n\lambda -\gamma^n R^n\right),
\end{align*}
since $\mathscr H^{n-1}(\mb S^{n-1})=n\omega_n$.
Moreover, by Markov's inequality and \eqref{eq:intfgammaR},
\begin{align*}
\mathscr L^n(v(\Theta_2\times [R,1]))& = \int_{\Theta_2\times [R,1]} f_{\gamma,R}(x) \d x \\
& = \mathscr H^{n-1}(\Theta_2)\frac{1-\gamma^n R^n}{n} 
\leq \frac{1-\gamma^n R^n}{n \lambda} \int_{\mb A(R,1)} |\p_r v(r\theta)| \d r \d \theta.
\end{align*}
Combining the last two estimates and choosing $\lambda=\frac{1-\gamma^n R^n}{2n}$, we find
$$\frac{1}{4} \int_{\mb A(R,1)} |\p_r u(r \theta)| \d r \d \theta 
= \frac{ n \omega_n (1-\gamma R)}{4} 
\leq \frac{n \omega_n(1-(\gamma R)^n)}{4}
\leq \int_{\mb A(R,1)} |\p_r v(r \theta)| \d r \d \theta$$
since $(\gamma R)^n<\gamma R<1$.
The conclusion follows.
\end{proof}

\begin{remark}\label{remark:genBC}
It is clear from the proof that the boundary condition $v=\tp{id}$ on $\p B_1(0)$ can be weakened to the requirement that
$
v(\theta)\in \mb S^{n-1}\tp{ for }
\mathscr H^{n-1}\tp{-a.e. } \theta\in \mb S^{n-1}.$
Note that this condition is independent of the representative of the equivalence class of $v\in W^{1,q}(B_1(0),\R^n)$.
The argument above carries through simply by replacing the set $\Theta$ with $\Theta\cap \{v(\theta)\in \mb S^{n-1}\}$.
\end{remark}

Combining Lemma \ref{lemma:perturbation} with Proposition \ref{prop:quasiminimizers}, we immediately obtain:

\begin{cor}\label{cor:blowup}
Let $1\leq p< q<\infty$ and $n\leq q$. For any $\e,\eta>0$ and any $f\in Z_p$,
there is a sequence $f_j\in B_{Z_p^\eta}(f,\e)$ such that $\tp{dist}_p(f_j, f)\to 0$ and 
$\mc E_q(f_j)\to \infty.$
\end{cor}

We are ready to prove our main result, which we restate here for convenience of the reader.

\begin{thm}
Fix $0<c<1$ and $1\leq p<\infty$ .
The set of those $f\in Z_p$ such that there is a solution $u\in \bigcup_{n\leq q,\, p<q} W^{1,q}(\Omega,\R^n)$ of \eqref{eq:jac} is meagre in $Z_p$.
\end{thm}

\begin{proof}
We focus on the case $\Omega = B_1(0)$ first.
Fix $\delta>0$ sufficiently small; in particular, $\delta<1-c$ suffices.
We replace $c$ with $c/(1-\delta)<1$ in the definition \eqref{eq:defZeta} of $Z_p^\eta$, in order to account for the slightly worse lower bound satisfied by the perturbations of Lemma \ref{lemma:perturbation}.

Let us first fix $q$ such that $n\leq q$ and $p<q$.
For $k\in \N$ let
$Y_k \equiv \{ f \in Z^\delta_p: \mc E_q(f)\leq k \}.$
Note that each $Y_k$ is closed in $Z_p^\delta$. Indeed, given a sequence $f_j\in Y_k$ such that $\tp{dist}_p(f_j, f)\to 0$ for some $f\in Z_p^\delta$, if $v_j$ are energy-minimal solutions corresponding to $f_j$, then from weak compactness there is a function $v\in W^{1,q}(B_1,\R^n)$ such that
 $v_j\w v$ in $W^{1,q}(B_1,\R^n)$. By the weak continuity of the Jacobian and weak lower semicontinuity of the $q$-Dirichlet energy, it follows that $\mc E_q(v)\leq \liminf_j \mc E_q(v_j)\leq k$.
 
We show that $Y_k$ has empty interior in $Z_p^\delta$. Indeed, take $f_0\in Y_k$ and consider an arbitrary $\e>0$: we claim that $B_{Z_p^\delta}(f_0,\e)\not \subset Y_k$.
For this, it suffices to show that $B_{Z_p}(f_0,\e)\not\subset Y_k$, since $ B_{Z_p}(f_0,\e)\subset B_{Z^\delta_p}(f_0,\e)$. Any such ball $B_{Z_p}(f_0,\e)$ contains an element $f_1$ which is in $Z_p$; thus, by replacing $f_0$ with $f_1$ and shrinking $\e$ if need be, we can assume without loss of generality that $f_0\in Z_p$. 
By Corollary \ref{cor:blowup}, there are $f_j \in B_{Z^\delta_p}(f_0,\e)\subset Y_k$ such that $\tp{dist}_p(f_j, f)\to 0$ but $j\leq \mc E_q(f_j)$. This proves the claim.

Since each set $Y_k$ is closed and has empty interior, the Baire Category Theorem implies that $\bigcup_{k\in \N} Y_k$ is meagre in $Z_p^\delta$. Equivalently, 
$\J (\tp{id}+W^{1,q}_0(\Omega,\R^n)) \tp{ is meagre in } Z_p^\delta.$ For $p\geq n$, we have that
$$\J\biggr(\tp{id}+\bigcup_{n\leq q,\, p<q} W^{1,q}_0(\Omega,\R^n)\biggr)
=
\bigcup_{j=1}^\infty \J \left(\tp{id}+W^{1,q_j}_0(\Omega,\R^n)\right), \qquad q_j\searrow p,$$
and the right-hand side is a countable union of meagre sets, hence meagre, so the conclusion follows. The case $p<n$ is identical.

We conclude the proof by considering the case where $\Omega$ is a general smooth and uniformly convex domain. There is a smooth diffeomorphism $\varphi\colon \overline \Omega\to \overline{B_1(0)}$ with $\J \varphi =\omega_n/|\Omega|$, see for instance \cite[Theorem 5.4]{Fonseca1992} for a more general version of this fact; the existence of such a map can also be deduced from \textsc{Caffarelli}'s regularity theory for the Monge--Ampère equation, see e.g.\ \cite[Theorem 3.3]{DePhilippis2014}. 
We can now identify every map $u\colon \Omega\to \R^n$ with a map $v\colon B_1(0)\to \R^n$ with the same regularity, through $v\equiv \varphi \circ u \circ \varphi^{-1}$ and moreover, when $x\in \p B_1(0)$, we have $v(x) = x$, as $\varphi(\p \Omega)=\p B_1(0)$. Since $\J v = \J u \circ \varphi^{-1}$, the general statement follows from the case $\Omega=B_1(0)$.
\end{proof}

{\footnotesize
\bibliographystyle{acm}
\bibliography{/Users/antonialopes/Dropbox/Oxford/Bibtex/library.bib}

\begin{thebibliography}{10}

\bibitem{Alberti2003}
{\sc Alberti, G., Baldo, S., and Orlandi, G.}
\newblock {Functions with prescribed singularities}.
\newblock {\em Journal of the European Mathematical Society 5}, 3 (2003),
  275--311.

\bibitem{Astala2009}
{\sc Astala, K., Iwaniec, T., and Martin, G.}
\newblock {\em {Elliptic Partial Differential Equations and Quasiconformal
  Mappings in the Plane (PMS-48)}}.
\newblock Princeton University Press, 2009.

\bibitem{Astala2012}
{\sc Astala, K., Iwaniec, T., Prause, I., and Saksman, E.}
\newblock {Burkholder integrals, Morrey's problem and quasiconformal mappings}.
\newblock {\em Journal of the American Mathematical Society 25}, 2 (2012),
  507--531.

\bibitem{Ball1977}
{\sc Ball, J.~M.}
\newblock {Convexity conditions and existence theorems in nonlinear
  elasticity}.
\newblock {\em Arch. Ration. Mech. Anal. 63}, 4 (1977), 337--403.

\bibitem{Ball1981a}
{\sc Ball, J.~M.}
\newblock {Global invertibility of Sobolev functions and the interpenetration
  of matter}.
\newblock {\em Proceedings of the Royal Society of Edinburgh: Section A
  Mathematics 88}, 3-4 (1981), 315--328.

\bibitem{Ball1982}
{\sc Ball, J.~M.}
\newblock {Discontinuous Equilibrium Solutions and Cavitation in Nonlinear
  Elasticity}.
\newblock {\em Philosophical Transactions of the Royal Society A: Mathematical,
  Physical and Engineering Sciences 306}, 1496 (1982), 557--611.

\bibitem{Bourgain2002}
{\sc Bourgain, J., and Brezis, H.}
\newblock {On the equation $\operatorname{div}Y=f$ and application to control
  of phases}.
\newblock {\em Journal of the American Mathematical Society 16}, 02 (2002),
  393--427.

\bibitem{Brenier1991}
{\sc Brenier, Y.}
\newblock {Polar factorization and monotone rearrangement of vector-valued
  functions}.
\newblock {\em Communications on Pure and Applied Mathematics 44}, 4 (1991),
  375--417.

\bibitem{Brezis2011b}
{\sc Brezis, H., and Nguyen, H.-M.}
\newblock {The Jacobian determinant revisited}.
\newblock {\em Inventiones mathematicae 185}, 1 (2011), 17--54.

\bibitem{Burago1998}
{\sc Burago, D., and Kleiner, B.}
\newblock {Separated nets in Euclidean space and Jacobians of biLipschitz
  maps}.
\newblock {\em Geometric And Functional Analysis 8}, 2 (1998), 273--282.

\bibitem{Caffarelli1992}
{\sc Caffarelli, L.~A.}
\newblock {The Regularity of Mappings with a Convex Potential}.
\newblock {\em Journal of the American Mathematical Society 5}, 1 (1992),
  99--104.

\bibitem{Caffarelli1996}
{\sc Caffarelli, L.~A.}
\newblock {Boundary Regularity of Maps with Convex Potentials--II}.
\newblock {\em The Annals of Mathematics 144}, 3 (1996), 453.

\bibitem{Carlier2012}
{\sc Carlier, G., and Dacorogna, B.}
\newblock {R{\'{e}}solution du probl{\`{e}}me de Dirichlet pour
  l'{\'{e}}quation du Jacobien prescrit via l'{\'{e}}quation de
  Monge-Amp{\`{e}}re}.
\newblock {\em Comptes Rendus Mathematique 350}, 7-8 (2012), 371--374.

\bibitem{Coifman1993}
{\sc Coifman, R.~R., Lions, P.~L., Meyer, Y., and Semmes, S.}
\newblock {Compensated compactness and Hardy spaces}.
\newblock {\em Journal de Math{\'{e}}matiques Pures et Appliqu{\'{e}}es 9}, 72
  (1993), 247--286.

\bibitem{Csato2012}
{\sc Csat{\'{o}}, G., Dacorogna, B., and Kneuss, O.}
\newblock {\em {The Pullback Equation for Differential Forms}}.
\newblock Birkh{\"{a}}user Boston, 2012.

\bibitem{Dacorogna1981}
{\sc Dacorogna, B.}
\newblock {A relaxation theorem and its application to the equilibrium of
  gases}.
\newblock {\em Archive for Rational Mechanics and Analysis 77}, 4 (1981),
  359--386.

\bibitem{Dacorogna1990}
{\sc Dacorogna, B., and Moser, J.}
\newblock {On a partial differential equation involving the Jacobian
  determinant}.
\newblock {\em Annales de l'Institut Henri Poincare (C) Non Linear Analysis 7},
  1 (1990), 1--26.

\bibitem{DePhilippis2014}
{\sc {De Philippis}, G., and Figalli, A.}
\newblock {The Monge–Amp{\`{e}}re equation and its link to optimal
  transportation}.
\newblock {\em Bulletin of the American Mathematical Society 51}, 4 (2014),
  527--580.

\bibitem{Faraco2020}
{\sc Faraco, D., and Guerra, A.}
\newblock {A short proof of Ornstein's non-inequality in $\mathbb R^{2\times
  2}$}.
\newblock {\em arXiv:2006.09060\/} (2020).

\bibitem{Fischer2019}
{\sc Fischer, J., and Kneuss, O.}
\newblock {Bi-Sobolev solutions to the prescribed Jacobian inequality in the
  plane with $L^p$ data and applications to nonlinear elasticity}.
\newblock {\em Journal of Differential Equations 266}, 1 (2019), 257--311.

\bibitem{Fonseca1995}
{\sc Fonseca, I., and Gangbo, W.}
\newblock {\em {Degree theory in analysis and applications}}.
\newblock Oxford University Press, 1995.

\bibitem{Fonseca1992}
{\sc Fonseca, I., and Parry, G.}
\newblock {Equilibrium configurations of defective crystals}.
\newblock {\em Archive for Rational Mechanics and Analysis 120}, 3 (1992),
  245--283.

\bibitem{GuerraKochLindberg2020b}
{\sc Guerra, A., Koch, L., and Lindberg, S.}
\newblock {A nonlinear open mapping principle and applications to the Jacobian
  equation}.
\newblock {\em In preparation\/} (2020).

\bibitem{Guerra2019}
{\sc Guerra, A., and Rai\cb{t}\u{a}, B.}
\newblock {Quasiconvexity, null Lagrangians, and Hardy space integrability
  under constant rank constraints}.
\newblock {\em arXiv:1909.03923\/} (2019).

\bibitem{Hencl2014a}
{\sc Hencl, S., and Koskela, P.}
\newblock {\em {Lectures on Mappings of Finite Distortion}}, vol.~2096 of {\em
  Lecture Notes in Mathematics}.
\newblock Springer International Publishing, Cham, 2014.

\bibitem{Hencl2002}
{\sc Hencl, S., and Mal{\'{y}}, J.}
\newblock {Mappings of finite distortion: Hausdorff measure of zero sets}.
\newblock {\em Mathematische Annalen 324}, 3 (2002), 451--464.

\bibitem{Hogan2000}
{\sc Hogan, J., Li, C., McIntosh, A., and Zhang, K.}
\newblock {Global higher integrability of Jacobians on bounded domains}.
\newblock {\em Annales de l'Institut Henri Poincare (C) Analyse Non Lineaire
  17}, 2 (2000), 193--217.

\bibitem{Hytonen2018}
{\sc Hyt{\"{o}}nen, T.~P.}
\newblock {The $L^p$-to-$L^q$ boundedness of commutators with applications to
  the Jacobian operator}.
\newblock {\em arXiv:1804.11167\/} (2018), 1--35.

\bibitem{Hytonen2019}
{\sc Hyt{\"{o}}nen, T.~P.}
\newblock {Of commutators and Jacobians}.
\newblock {\em arXiv:1905.00814\/} (2019), 1--9.

\bibitem{Iwaniec1997}
{\sc Iwaniec, T.}
\newblock {Nonlinear commutators and jacobians}.
\newblock {\em The Journal of Fourier Analysis and Applications 3}, S1 (1997),
  775--796.

\bibitem{Iwaniec2001}
{\sc Iwaniec, T., and Martin, G.}
\newblock {\em {Geometric Function Theory and Non-linear Analysis}}.
\newblock Clarendon Press, 2001.

\bibitem{Iwaniec1999b}
{\sc Iwaniec, T., and Verde, A.}
\newblock {On the Operator $\mathcal L(f) = f \log |f|$}.
\newblock {\em Journal of Functional Analysis 169}, 2 (1999), 391--420.

\bibitem{Kirchheim2016}
{\sc Kirchheim, B., and Kristensen, J.}
\newblock {On Rank One Convex Functions that are Homogeneous of Degree One}.
\newblock {\em Archive for Rational Mechanics and Analysis 221}, 1 (2016),
  527--558.

\bibitem{Koumatos2015}
{\sc Koumatos, K., Rindler, F., and Wiedemann, E.}
\newblock {Differential Inclusions and Young Measures Involving Prescribed
  Jacobians}.
\newblock {\em SIAM Journal on Mathematical Analysis 47}, 2 (2015), 1169--1195.

\bibitem{Lindberg2020}
{\sc Lindberg, S.}
\newblock {A note on the Jacobian problem of Coifman, Lions, Meyer and Semmes}.
\newblock {\em In preparation\/}.

\bibitem{Lindberg2015}
{\sc Lindberg, S.}
\newblock {On the Jacobian equation and the Hardy space $\mathcal H^1(\mathbb
  C)$}.
\newblock {\em Annales Academiae Scientiarum Fennicae Mathematica
  Dissertationes 160\/} (2015), 1--64.

\bibitem{Lindberg2017}
{\sc Lindberg, S.}
\newblock {On the Hardy Space Theory of Compensated Compactness Quantities}.
\newblock {\em Archive for Rational Mechanics and Analysis 224}, 2 (2017),
  709--742.

\bibitem{McMullen1998}
{\sc McMullen, C.~T.}
\newblock {Lipschitz maps and nets in Euclidean space}.
\newblock {\em Geometric and Functional Analysis 8}, 2 (1998), 304--314.

\bibitem{Mityagin1958}
{\sc Mityagin, B.~S.}
\newblock {On second mixed derivative}.
\newblock In {\em Doklady Akademii Nauk}, vol.~123. Russian Academy of
  Sciences, 1958, pp.~606--609.

\bibitem{Moser1965}
{\sc Moser, J.}
\newblock {On the Volume Elements on a Manifold}.
\newblock {\em Transactions of the American Mathematical Society 120}, 2
  (1965), 286--294.

\bibitem{Muller1990a}
{\sc M{\"{u}}ller, S.}
\newblock {Det = det. A remark on the distributional determinant}.
\newblock {\em Comptes rendus de l'Acad{\'{e}}mie des sciences. S{\'{e}}rie 1,
  Math{\'{e}}matique 311}, 1 (1990), 13--17.

\bibitem{Muller1990}
{\sc M{\"{u}}ller, S.}
\newblock {Higher integrability of determinants and weak convergence in $L^1$.}
\newblock {\em Journal f{\"{u}}r die reine und angewandte Mathematik (Crelles
  Journal) 1990}, 412 (1990), 20--34.

\bibitem{Ornstein1962}
{\sc Ornstein, D.}
\newblock {A non-inequality for differential operators in the $L_1$ norm}.
\newblock {\em Archive for Rational Mechanics and Analysis 11}, 1 (1962),
  40--49.

\bibitem{Riviere1996}
{\sc Rivi{\`{e}}re, T., and Ye, D.}
\newblock {Resolutions of the prescribed volume form equation}.
\newblock {\em Nonlinear Differential Equations and Applications NoDEA 3}, 3
  (1996), 323--369.

\bibitem{Stein1969}
{\sc Stein, E.}
\newblock {Note on the class $L\log L$}.
\newblock {\em Studia Mathematica 32}, 3 (1969), 305--310.

\bibitem{Sverak1988}
{\sc {\v{S}}ver{\'{a}}k, V.}
\newblock {Regularity properties of deformations with finite energy}.
\newblock {\em Archive for Rational Mechanics and Analysis 100}, 2 (1988),
  105--127.

\bibitem{Trudinger2009}
{\sc Trudinger, N.~S., and Wang, X.-J.}
\newblock {On the second boundary value problem for equations of
  Monge-Amp{\`{e}}re type.}
\newblock {\em Ann. Scuola Norm. Sup. Pisa Cl. Sci. 8}, 1 (2009), 143--174.

\bibitem{Viera2018}
{\sc Viera, R.}
\newblock {Densities non-realizable as the Jacobian of a 2-dimensional
  bi-Lipschitz map are generic}.
\newblock {\em Journal of Topology and Analysis 10}, 4 (2018), 933--940.

\bibitem{Vodopyanov1976}
{\sc Vodop'yanov, S.~K., and Gol'dshtein, V.~M.}
\newblock {Quasiconformal mappings and spaces of functions with generalized
  first derivatives}.
\newblock {\em Siberian Mathematical Journal 17}, 3 (1977), 399--411.

\bibitem{Ye1994}
{\sc Ye, D.}
\newblock {Prescribing the Jacobian determinant in Sobolev spaces}.
\newblock {\em Annales de l'Institut Henri Poincare (C) Non Linear Analysis
  11}, 3 (1994), 275--296.

\end{thebibliography}
}

\end{document}